\documentclass[11pt]{amsart}

\newcommand{\private}[1]{}
\usepackage{amssymb, amsmath, amscd, amsthm, color, epsfig, url, graphicx}

\usepackage[all]{xy}          % for xy-pic pictures
\xyoption{dvips}              % for xy-pic pictures

\makeatletter
\renewcommand\l@subsection{\@tocline{2}{0pt}{2pc}{5pc}{}}
\makeatother

\newcommand{\R}{{\mathbb R}}

\newcommand{\Emb}{\operatorname{Emb}}

\newcommand{\K}{{\mathcal{K}}}

\newcommand{\Lk}{{\mathcal{L}}}

\newcommand{\Br}{{\mathcal{B}}}

\newcommand{\LD}{{\mathcal{LD}}}

\newcommand{\BD}{{\mathcal{BD}}}

\newcommand{\ord}{\operatorname{ord}}
\newcommand{\degree}{\operatorname{deg}}

\theoremstyle{plain}
\newtheorem{thm}{Theorem}[section]
\newtheorem{prop}[thm]{Proposition}

\newtheorem{conj}[thm]{Conjecture}

\theoremstyle{definition}
\newtheorem{defin}[thm]{Definition}

\theoremstyle{remark}
\newtheorem{rem}[thm]{Remark}

\newcommand{\refT}[1]{Theorem~\ref{T:#1}}

\newcommand{\refP}[1]{Proposition~\ref{P:#1}}
\newcommand{\refD}[1]{Definition~\ref{D:#1}}

\begin{document}

%%%%%%%%%%%%%%%%%%%%%%%%%%%%%%%%%%%%%%%%%%%%%%

\title[Cohomology of links via configuration space integrals]{On the cohomology of spaces of links and braids via configuration space integrals}

%%%%%%%%%%%%%%%%%%%%%%%%%%%%%%%%%%%%%%%%%%%%%%

\author{Ismar Voli\'c}
\address{Department of Mathematics, Wellesley College, Wellesley, MA}
\email{ivolic@wellesley.edu}
\urladdr{http://palmer.wellesley.edu/\~{}ivolic}


\thanks{The author was supported in part by the National Science Foundation grant DMS 0805406.}

%%%%%%%%%%%%%%%%%%%%%%%%%%%%%%%%%%%%%%%%%%%%%%%%%%%%%%%%%%%%%%%%%%%

\begin{abstract}
%{\bf Version: \today}
We study the cohomology of spaces of string links and braids in $\R^n$ for $n\geq 3$ using configuration space integrals.  For $n>3$, these integrals give a chain map from certain diagram complexes to the deRham algebra of differential forms on these spaces.  For $n=3$, they produce all finite type invariants of string links and braids. 
\end{abstract}

\maketitle

\tableofcontents

\parskip=4pt
\parindent=0cm

%%%%%%%%%%%%%%%%%%%%%%%%%%%%%%%%%%%%%%%%%%%%%%%%%%%%%%%%%%%%%%%%%%%%%%%%%%%%%%%%%%%%%%%%%%%%%%%%%%%%%%%%

\section{Introduction}\label{S:Intro}

%%%%%%%%%%%%%%%%%%%%%%%%%%%%%%%%%%%%%%%%%%%%%%%%%%%%%%%%%%%%%%%%%%%%%%%%%%%%%%%%%%%%%%%%%%%%%%%%%%%%%%%%

The goal of this paper is to establish and study configuration space integrals for spaces of string (or long) links\footnote{The reader should keep in mind that when we say string links we do not mean {\it homotopy} string links as is sometimes the case in literature; see \refD{MainDef}.} and pure braids in $\R^n$ for $n\geq 3$. 
%(we will refer to the two spaces as ``spaces of links" unless otherwise specified).  
This should be thought of as a generalization of the study of such integrals in the case of knots and long knots.
Namely, in \cite{CCRL}, Cattaneo, Cotta-Ramusino, and Longoni define a cochain map
\begin{equation}\label{E:CCRLMap}
\mathcal{D}\longrightarrow \Omega^*\K^n
\end{equation}
between a certain diagram complex $\mathcal{D}$ and the the deRham complex of space of knots (and long knots)  $\K^n$ in $\R^n$ for $n>3$
% and use this to show, among other things, that one finds cohomology classes of $\K^n$ in arbitrarily high degrees.  The map is defined 
 via configuration space integrals, defined essentially as the following:  A diagram $\Gamma\in\mathcal{D}$ is used as a prescription for pulling back a product of volume forms on $S^{n-1}$ to a space $C[k,s;\K^n, \R^n]$, which is a bundle over $\K^n$ whose fiber over $K\in\K^n$ is the compactified configuration space of $k+s$ points in $\R^n$, first $k$ of which are constrained to lie on $K$.  This form can then be pushed forward to $\K^n$, i.e.~integrated along the fiber of the map $C[k,s;\K^n, \R^n]\to\K^n$.

Cattaneo, Cotta-Ramusino, and Longoni have further shown that, for any $k$, space of knots has cohomology classes in arbitrarily high degrees in \cite{CCRL:Struct} by studying certain algebraic structures on $\mathcal{D}$ that correspond to those in the cohomology ring of spaces of knots.  Longoni also proved in \cite{Long:Classes} that some of these classes arise from non-trivalent diagrams.  An analog for long knots was given in \cite{Sakai:Nontrivalent} by Sakai who combined the configuration space integrals with Budney's action of the little discs operad on $\K^n$ \cite{Bud:LCLK}.  Sakai has further defined configuration space integrals for embeddings of long planes $\Emb(\R^j, \R^n)$ \cite{Sakai:BTLongPlanes}.  In another recent work, Koytcheff \cite{Koyt:HomotopyBT}  has reinterpreted configuration space integrals in a homotopy-theoretic way that does not use integration.  The advantage is that he obtains integral, rather than real, classes in the cohomology of knots.

The results in \cite{CCRL} are in many respects a culmination of work begun by Guadagnini, Martellini, and Mintchev \cite{GMM} and Bar-Natan \cite{BN:Thesis} and further developed by Bott and Taubes \cite{BT} (in literature, configuration space integrals are often referred to as ``Bott-Taubes integrals") and D. Thurston \cite{Thurs}.  Kontsevich had also studied similar integration techniques in \cite{K:OMDQ}. The
focus of Bott-Taubes and D. Thurston's work was knots in the classical dimension $n=3$, and  the main results in knot theory which uses configuration space integrals is that they represent a universal finite type knot invariant, due to D. Thurston (see also \cite{V:SBT}). 

In the present work, we generalize both the results of Cattaneo-Cotta-Ramusino-Longoni and D. Thurston.  To elaborate, we first need to define the spaces in question precisely.

\begin{defin}\label{D:MainDef}  For $n\geq 3$, define
\begin{itemize}
\item  \emph{Space of string links, or long links, of $m$ components in $\R^n$}, denoted by $\Lk_{m}^{n}$, to be the space of embeddings $\Emb(\coprod_{m}\R, \R^n)$ of the disjoint union of $m$ copies of $\R$ in $\R^n$ that agree, outside of $\{-1,1\}\times \R^{n-1}$, with some fixed linear embedding of $\coprod_{m}\R$ in $\R^n$ which maps each $\R$ to a line parallel to the line $(x,0,0,...,0)$ in $\R^n$;
%\item  \emph{space of homotopy string links, or homotopy links, of $m$ components in $\R^n$}, denoted by $\HLk_{m}^{n}$, to be the space of embeddings $\Emb(\coprod_{m}\R, \R^n)$ of the disjoint union of $m$ copies of $\R$ in $\R^n$ modulo the relation that two embeddings are equivalent if there exist a homotopy between them throughout which the images of the copies of $\R$ are always disjoint;
\item \emph{Space of ``flattened" pure braids of $m$ components in $\R^n$}, denoted by $\Br_{m}^{n}$, to be the subspace of the space of embeddings $\Emb(\coprod_{m}\R, \R^n)$ defined above determined by those embeddings whose normalized derivative on each of the copies of $\R$ 
\begin{enumerate}
\item has positive first component, and further
\item misses some fixed small neighborhoods $U_N$ and $U_S$ of the north and south poles, respectively, of the standard sphere in $ \R^{n}$.
\end{enumerate}

%\item  A subspace $\Br_{m}^{n}$ of the space of pure braids given by those pure braids whose normalized tangent always misses some fixed small neighborhoods $U_N$ and $U_S$ of the north and south poles, respectively, of the standard sphere in $ \R^{n}$.
%of the disjoint union of $m$ copies of $\R$ in $\R^n$ (so for each $t\in\R$ and $f\in \Emb(\coprod_{m}\R, \R^n)$, $f(t, t, ..., t)$ is a configuration of $m$ distinct points in the plane $\{t\}\times\R^{n-1}$);
\end{itemize}
%
%such that all embeddings and homotopy link maps agree, outside of $\{-1,1\}\times \R^{n-1}$, with some fixed linear embedding of $\coprod_{m}\R$ in $\R^n$ which maps each $\R$ to a line parallel to the line $(x,0,0,...,0)$ in $\R^n$.   

\end{defin}
Note that if one takes away condition $(2)$, then the second space is the standard space of pure braids of $m$ components.  Its subspace $\Br_{m}^{n}$ is very closely related to it since it is immediate that any pure braid is isotopic to a braid in $\Br_{m}^{n}$.  The isotopy is given by ``flattening" the braid as much as necessary so that the tangent vector always misses $U_N$ and $U_S$.  This is of course not always possible with links in $\Lk_{m}^{n}$ because condition $(1)$ is not satisfied.

We will usually refer to ``long links" simply as ``links" and to elements of $\Br_{m}^{n}$ as ``braids'' througout.  In fact, we will often simply say ``links" for either space and will make the distinction clear when necessary.  Both links and braids are interesting in its own right and have been studied extensively.  More about our motivation for studying them is given below.

Our main results can now be summarized as 
\begin{thm}\label{T:MainThmIntro} For each $o>0$, $d\geq 0$, and $n> 3$,
there exist  diagram complexes $\LD^{o,d}$
%, $\HLD^{o,d}$, 
and $\BD^{o,d}$ and cochain maps 
\begin{align*}
%I_{\Lk}\colon  
\LD^{o,*} & \longrightarrow \Omega^{(n-3)o+*}(\Lk_{m}^{n}) \\
%I_{\HLk}\colon  
%\HLD^{o,*} & \longrightarrow \Omega^{(n-3)o+*}(\HLk_{m}^{n}) \\
%I_{\Br}\colon 
 \BD^{o,*} & \longrightarrow \Omega^{(n-3)o+*}(\Br_{m}^{n})
\end{align*}
given by configuration space integrals.  For $n=3$, these maps give isomorphisms between $H^0(\LD^{o,0})$
%, $H^0(\HLD^{o,0})$, 
and $H^0(\BD^{o,0})$ and finite type invariants of order $o$ of those spaces.
\end{thm}

One immediately has the following 
\begin{conj}  The cochain maps from \refT{MainThmIntro} are quasi-isomorphisms for $n>3$. 
\end{conj}

The evidence for this conjecture comes from the case of knots, where it seems likely that the bicomplex of diagrams is quasi-isomorphic to the $E_1$ page of the Vassiliev spectral sequence \cite{Vass:Cohom}.  The latter is known to converge to the cohomology of knots for $n>3$ and collapses at $E_2$ \cite{LTV:Vass}.  Vassiliev-like spectral sequences for spaces of links have been constructed in \cite{MV:Multi}, and the next step   should be to show that they collapse as in the case of knots.

%We do not have a sense of whether the same result should hold for $n=3$ in degree zero.  If it did, this would in particular mean that finite type invariants separate knots and links.  This question is one of the motivations for the work here.

The part of \refT{MainThmIntro} which states that there exist cochain maps is proved later as Theorem \ref{T:MainThm}.
  The result about finite type invariants is \refT{UniversalFiniteType}.  The diagram complexes, defined in section \ref{S:Diagrams}, should be thought of as just a convenient way of keeping track of integrals along the interior and the boundary of compactified configuration spaces.  Before we define the integrals, we will have to do away with one of the more technical aspects of the story.  Namely, in order for integrals to converge, configuration spaces have to be compactified.  This construction is reviewed in section \ref{S:Compactification}, as is the construction of certain spaces analogous to $C[k,s;\K^n, \R^n]$ which fiber over  our link and braid spaces.  We then define the integrals in section \ref{S:Forms} and show that they produce chain maps given in \refT{MainThmIntro}  when $n>3$.  This directly generalizes  the main results from \cite{CCRL} to spaces of links and braids.  The argument is essentially to show that the configuration space integrals vashing along most of the codimension one boundary of the compactified configuration spaces.  

The difference for $n=3$ comes from the fact that the integrals along some faces do not necessarily vanish for all diagrams.  However, this can be fixed 
%in some sense since one has a good handle (up to constant) on what the anomalous integral must be and  an appropriate term which cancels it can then be subtracted.  This is the subject of section \ref{S:AnomalousFaces}.  \refT{n=3CochainMap} provides this fix, giving a new map $\LD^{o,*} \to \Omega^{*}(\Lk_{m}^{3})$ which is now a morphism of cochain complexes.
when $d=0$ (the interesting degree which contains link invariants), and this is done in \refT{n=3CochainMap}.  The diagram complexes  become complexes of \emph{trivalent diagrams} which have been received much attention in recent years \cite{BN:Vass, Long:Classes} and provide a bridge to finite type theory.  We recall some facts about these in section \ref{S:TrivalentDiagrams} and  in section \ref{S:UniversalFiniteType} prove the second part of \refT{MainThmIntro} after reviewing the basics of finite type theory.  This result is the analog of D. Thurston's results from \cite{Thurs} about configuration space integrals classifying finite type invariants for knots.

\vskip 6pt

This purpose of this paper is two-fold:  One the one hand, we provide generalizations to links and braids of the many results involving configuration space integrals for knots, but we do so in a fashion which streamlines and brings together the often disparate literature on the subject.  In particular, the cases of classical knots and knots in codimension $>2$ have often been treated differently and from separate points of view, and our goal is to bring those cases together in this work.

Secondly, we wish to apply the results of this paper in the setting of manifold calculus of functors.  Namely, B. Munson  and the author have initiated in \cite{MV:Links, MV:Multi} the study of spaces of links by defining certain multi-towers of spaces whose stages represent approximations to spaces of links.  In particular, these multi-towers are expected to classify all finite type invariants of links and braids.  The analogous result for knots was established in \cite{V:FTK}.  The key in showing such a classification statement is the extension of configuration space integrals from knots to links and braids, which is precisely what is done in this paper.  This extension can then be further  modified so that the target of the integrals is stages of the multi-towers (for knots, this was done in \cite{V:IT}).  This modification will be addressed in a future paper.  It is expected that factoring configuration space integrals through stages of the Taylor multi-towers will also lead to a new proof that finite type invariants separate braids and homotopy string links as well as new connections to and generalizations of Milnor invariants.

%After this, the goal is to reprove, in the functor calculus setting, that finite type invariants in fact separate homotopy links and braids (i.e.~that given two different links, there is a finite type invariant which can tell them apart).  These results were originally proved in \cite{HabLin-Classif}  and  \cite{BN:HoLink, Kohno:Vassiliev}, respectively.  The hope is that this will point the way toward proving or disproving the same result for spaces of classical links and knots.   The question of separation of knots and links by finite type invariants has been one of the most interesting problems in knot and link theory in recent years. One reason the strategy described here might be fruitful in resolving it is that, unlike any other machinery known to us, calculus of functors is a general theory that applies in exactly the same way to knots and both spaces from \refD{MainDef}. 

Here are some further questions that immediately arise from the results in this paper:

\textbullet\ If one thinks of $\Br_{m}^{n}$ as $\Omega C(m, \R^{n-1})$, the loop space of the configuration space of $m$ points in $\R^{n-1}$ (to pass to this model, one would require a constant, rather than just positive, derivative in \refD{MainDef}), then one should have another interesting connection between the work here and the work done on the homology of $\Br_{m}^{n}$  by F. Cohen and Gitler \cite{CG} and Kohno \cite{Kohno:LoopsFiniteType}.  In particular, Kohno constructs invariants on $\Br_{m}^{n}$ via integrals which are likely related to ours.  The difference comes from the fact that  we are forced to think of braids as a subspace of a certain space of embeddings, mainly because of \refP{Bundles} and the definitions surrounding it.  One should be able to reconcile the two points of view and in particular simplify the braid diagram complex from section \ref{S:Diagrams} so that the connected components of diagrams have external vertices only along ``vertical slices".  This would reflect the fact that braids can be parametrized with a single loop parameter. If this were possible, configuration space integrals for braids would simplify greatly and the anomaly (see section \ref{S:AnomalousFaces}) would in particular disappear.

\textbullet\ One should also be able to define configuration space integrals for the space of \emph{homotopy} string links \cite{HabLin-Classif, Milnor-Mu}.  The problem, however, is that the most natural definition of this space is as a subspace of the space of immersions of $\coprod\R$ in $\R^n$, in which case it is not clear that the analog of the compactified configuration spaces from section \ref{S:Compactification} for homotopy string links would be a manifold with corners.
%
%\textbullet\ Anomaly for braids is same as for unknot.

\textbullet\ It is known that finite type invariants separate braids \cite{Kohno:LoopsFiniteType, BN:HoLink}.  It should be possible to reprove this result using configuration space integrals defined here. 

%What about defining homotopy links as subspace of PA maps?  Maybe then the pullback thing is a PA bundle so that pushforward makes sense?  But then links aren't a subspace since there are obviously embeddings that are not PA.  So define everything as PA?  Links as well?  Polynomial embeddings?  Are spaces of all embeddings and polynomial embeddings homotopy equivalent? 

%\footnote{What about all that stuff Catt et al do, like injective in degree 0, infinitely high-dimensional classes, etc?}

%%%%%%%%%%%%%%%%%%%%%%%%%%%%%%%%%%%%%%%%%%%%%%%%%%%%%%%%%%%%%%%%%%%%%%%%

\section{Diagram complexes}\label{S:Diagrams}

%%%%%%%%%%%%%%%%%%%%%%%%%%%%%%%%%%%%%%%%%%%%%%%%%%%%%%%%%%%%%%%%%%%%%%%%

%As mentioned above, complexes of diagrams defined in this section are simply an organizational device for keeping track of the configuration space integrals.

Given integers $n\geq 3$  and $m\geq 1$, we will consider connected diagrams $\Gamma$ consisting of $m$ oriented line segments labeled 1, 2, ..., $m$, with some number of vertices on or off them (see Figure \ref{Fig:DiagramExample}).   A vertex lying on a segment will be called \emph{external} and will otherwise be called \emph{internal}.    Internal vertices are at least trivalent.  Each segment always has two  external vertices at its endpoints.  External vertices on the $j$th segment are labeled $v^{j}_1, v^{j}_2,...$ in linear order.
% with respect to the subscript index.  
Internal vertices are also labeled.

\begin{rem}
This terminology comes from diagrams associated to ordinary closed knots where there is only one segment, drawn as a circle.  Vertices that are not on the circle are usually drawn inside it, and are labeled ``internal", while those on the circle are  labeled ``external".
\end{rem}

Vertices may be joined by \emph{edges}.  We identify four types of edges:

\begin{itemize}
\item  \emph{internal edge}, connecting two internal vertices;
\item  \emph{mixed edge}, connecting an internal vertex and an external vertex;
\item  \emph{chord}, connecting two external vertices;
\item  \emph{loop}, connecting an external vertex to itself.
\end{itemize}

From now on, when we say ``edge", we will mean any of the above types of edges unless otherwise specified.

We also identify 
\begin{itemize}
\item  \emph{arcs}, which are parts of line segments between two consecutive external vertices.
\end{itemize}

We will denote an edge or an arc with endpoints $a$ and $b$ by $(a,b)$.

A diagram may not contain an edge connecting an internal vertex to itself.  All edges are oriented.  Further, 
\begin{itemize}
\item \emph{for $n$ even}, edges are labeled;
\item \emph{for $n$ odd},  edges are oriented.
%internal vertices are labeled and.
%and half-edges of each loop are ordered.
\end{itemize}

A \emph{connected component} of a diagram is a subset of its vertices and edges which is not connected by an edge to any other part of the diagram (we disregard the arcs when identifying connected components). 

Given a diagram $\Gamma$, let 
\begin{itemize}
\item  $|e|=$ number of edges of $\Gamma$;
\item  $|v_{ext}|=$ number of external vertices of $\Gamma$;
\item  $|v_{int}|=$ number of internal vertices of $\Gamma$.
\end{itemize}

\begin{defin}
Define the \emph{order} and \emph{degree} of $\Gamma$ to be 
\begin{align*}
\ord\Gamma & =  |e|-|v_{int}|\\
\degree\Gamma & =   2|e|-3|v_{int}|-|v_{ext}|.
\end{align*}
\end{defin}
%\footnote{Why do we have order?  Pascal and I don't have it.}

\begin{defin}\label{D:DiagramSpaces}
Define 
%\begin{itemize}
%\item 
$\LD^{o,d}_{even}$ and $\BD^{o,d}_{even}$ 
%and $\HLD^{o,d}_{even}$
%, and $\BD^{o,d}_{even}$ 
(resp. $\LD^{o,d}_{odd}$ and $\BD^{o,d}_{odd}$)
% and $\HLD^{o,d}_{odd}$)
%, and $\BD^{o,d}_{odd}$)
 to be real vector spaces generated by diagrams $\Gamma$  described above for $n$ even (resp. odd) of order $o$ and degree $d$ 
%\item $\HLD^{o,d}_{even}$ (resp.~$\HLD^{o,d}_{even}$) to be the real vector space generated by those diagrams $\Gamma$  
%%described above for $n$ even (resp.~odd) of order $o$ and degree $d$ 
%which in addition
%\begin{itemize}
%\item do not contain a component whose external vertices lie on a single line segment and which is connected when the line segment is disregarded;
%\end{itemize}
%
%\item $\BD^{o,d}_{even}$ (resp. $\BD^{o,d}_{odd}$) to be the real vector space generated by those diagrams $\Gamma$  
%described above for $n$ even (resp.~odd) of order $o$ and degree $d$ 
%which in addition 
%to the condition in the previous item
%\begin{itemize}
%\item 
%do not contain a connected component with more than one external vertex on a single line segment;
%\item 
%do not contain loops or chords connecting external vertices on a single segment. 
%\item satisfy that whenever a connected component has external vertices $v^{i}_{a}$ and $v^{j}_{b}$ and another has external vertices $v^{i}_{a'}$ and $v^{j}_{b'}$, it must be that either $a<a'$ and $b<b'$ or $a>a'$ and $b>b'$;
%
%
%\end{itemize}
%
%\end{itemize}
modulo subspaces generated by the relations
\begin{enumerate}
\item  If $\Gamma$ contains more than one edge between two vertices, then $\Gamma=0$;
%\item  If $\Gamma$ contains an edge connecting a vertex to itself, then $\Gamma=0$;
%\item  If $\Gamma$ contains an chord connecting $v^{j}_{a}$ and $v^{j}_{a+1}$ for any $j$ and $a$, then $\Gamma=0$;
\item  If $n$ is odd  and $\Gamma$ and $\Gamma'$ differ by a permutation of the internal vertices or ordering of edges,
% or change of ordering of half-edges of some loops, 
then $\Gamma=(-1)^{\sigma}\Gamma'$ in $\LD^{o,d}_{odd}$ and $\BD^{o,d}_{odd}$, where $\sigma$ is the sum of the order  of the permutation and  the number of edges whose orientation is different;
% and the number of loops with different half-edge ordering;  
\item  If $n$ is even and $\Gamma$ and $\Gamma'$ differ by a permutation of the edge labels, then $\Gamma=(-1)^{\sigma}\Gamma'$ in $\LD^{o,d}_{even}$ and $\BD^{o,d}_{even}$, where $\sigma$ is the order of the permutation of the edges; 
\end{enumerate}
and additionally in the case of $\BD^{o,d}_{even}$ and $\BD^{o,d}_{odd}$,
\begin{itemize}
\item[(4)] If $\Gamma$ contains a loop, or a chord connecting external vertices on a single segment, then $\Gamma=0$.
\end{itemize}

%In addition, in $\HLD^{o,d}_{even}$, $\BD^{o,d}_{even}$, $\HLD^{o,d}_{odd}$, and $\BD^{o,d}_{odd}$, we also have
%\begin{enumerate}
%\item[(4)]  If $\Gamma$ has a connected component (meaning it is connected if line segments are disregarded) whose external vertices lie on a single segment, then $\Gamma=0$. 
%\end{enumerate}

%In addition, in $\BD^{o,d}_{even}$ and $\BD^{o,d}_{odd}$, we also have
%\begin{enumerate}
%\item [(4)]  If $\Gamma$ has a a connected component with more than one external vertex on a single segment, then $\Gamma=0$.

%\item[(4)]  If $\Gamma$ has a 
%%chord between two external vertices on the same segment or if it has a 
%loop, then $\Gamma=0$; 
%%\item[(5)]  If $\Gamma$ has two chords as in either of the diagrams pictured in Figure \ref{Fig:BraidDiagram}
%%\footnote{{\bf To the editor:}  If this paper is accepted for publication, I will redraw all figures in Xfig.}
%%, then $\Gamma=0$.

%\end{enumerate}

Define $\LD_{even}$, $\BD_{even}$, $\LD_{odd}$, and $\BD_{odd}$ to be direct sums of vector spaces above for all $o$ and $d$.

\end{defin}

An example of a diagram in $\LD_{odd}$ is given in Figure \ref{Fig:DiagramExample}.  In case when there is only one line segment, this gives precisely the complex associated to knots defined in \cite{CCRL}.  The last relation corresponds to the fact that braids always ``flow" in one direction.
% and \ref{Fig:DiagramExample2}.

When there is no danger of confusion, we will drop the subscripts and refer to our diagram spaces as $\LD$ and $\BD$ and will only make comments about parity when necessary.

% 
%Note that diagrams in $\BD$ should be thought of as having components whose external vertices are contained in ``vertical slices".  

%%\raisebox{-.15cm}{
%\begin{figure}
%  \centering
%  \fbox{\includegraphics[height=50mm]{DiagramExample.jpg}}
%  \label{Fig:DiagramExample}
%  \caption{Example of an element of $\LD_{odd}$ for $m=3$.  This is not an element of $\BD^{o,d}_{odd}$ because of the loops and chords $(v_1^5, v_1^5)$, $(v_3^3, v^3_3)$, $(v^1_5, v^1_6)$, $(v^2_1, v^2_3)$, and $(v^3_2, v^3_3)$.
%  }
%  \label{Fig:DiagramExample}
%\end{figure}

\begin{figure}[h]
\input{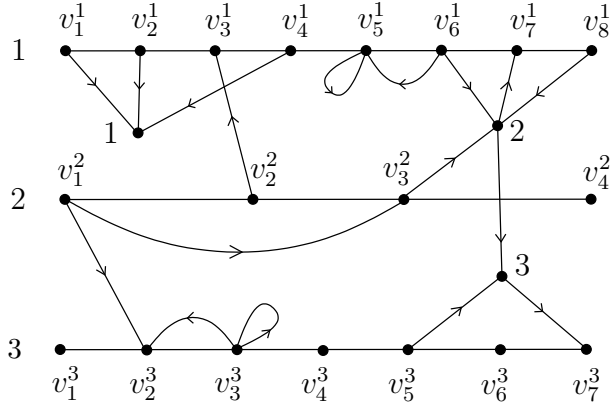}
\caption{Example of an element of $\LD_{odd}$ for $m=3$.  This is not an element of $\BD_{odd}$ because of the loops and chords $(v_1^5, v_1^5)$, $(v_3^3, v^3_3)$, $(v^1_5, v^1_6)$, $(v^2_1, v^2_3)$, and $(v^3_2, v^3_3)$.}
\label{Fig:DiagramExample}
\end{figure}

%%}

%

%%%\raisebox{-.15cm}{
%\begin{figure}
%  \centering
%  \fbox{\includegraphics[height=50mm]{DiagramExample2.jpg}}
%  \label{Fig:DiagramExample2}
%  \caption{Example of an element of $\BD^{o,d}_{odd}$ for $m=3$.}
%  \label{Fig:DiagramExample2}
%\end{figure}
%%%}

The coboundary operator will be defined via contraction of edges.  More precisely, let $e$ be a internal edge, a mixed edge, or an arc in a diagram $\Gamma$ and define $\Gamma/e$ to be the graph obtained by contracting $e$.  The labels in the new diagram are as follows:

\begin{itemize}
\item Orientations of edges other than $e$ remain unchanged;
\item The vertex that remains after contraction retains the higher of the two endpoint labels;
\item If a vertex (resp. edge) has a label higher than the label of the vertex  that remains after contraction (resp. label of $e$), its label is reduced by one. 
\end{itemize}

\begin{defin}  
Define the differential $\delta$ on $\LD$ and $\BD$ as a linear extension of \begin{equation}
\delta(\Gamma)= \sum_{
\text{internal edges, mixed edges, or arcs $e$ of $\Gamma$}}(-1)^{\epsilon(e)}\Gamma/e.
\end{equation} 
where
 $\epsilon(e)$ is a sign given as follows:

\begin{itemize}

\item Suppose $n$ is odd and $e$ in an edge or an arc, or suppose $n$ is even and $e$ is an arc (but not an edge).  Suppose $e$ connects vertex $i$ to vertex $j$ according to the orientation of the edge or arc.  Then
\begin{equation}\label{E:OddSign}
\epsilon(e)=\begin{cases}
(-1)^{j}, & j>i,  \\
(-1)^{i+1},  & j<i.
\end{cases}
\end{equation}

\item Suppose $n$ is even and $e$ is an edge.  Then 
\begin{equation}\label{E:EvenSign}
\epsilon(e)=(-1)^{(\text{label of $e$})+|v_{ext}|+1}. 
\end{equation}

\end{itemize}

\end{defin}

 This differential will later correspond precisely to Stokes' Theorem and 
integration over faces of compactified configuration spaces.

%The last bit of structure on our diagram spaces is the multiplication, which can be defined if a number of external vertices is fixed.  It is given by 
%superimposing, i.e.~by aligning the external 
%vertices.  The labels in $\Gamma\cdot \Gamma'$ are obtained by shifting the internal  
%vertex indices of $\Gamma'$ by the number of internal vertices of $\Gamma$ and indices of all its edges by
%the total number
%of edges of $\Gamma$.  With this convention, and 
%using the orientation relations, it is not hard to see that 
%\begin{equation}
%\Gamma\cdot \Gamma'=
%(-1)^{\vert\Gamma\vert\vert\Gamma'\vert}\Gamma'\cdot \Gamma.
%\end{equation}
%\footnote{It seems the degree as defined here is incompatible with the superimposition product.  Further, to define this product, we have to fix the number of external vertices, which B-T integrals don't like.}

\begin{thm}\label{T:DiagramsAreComplexes}
The map $\delta$ is well-defined and gives a coboundary operator on $\LD$ and $\BD$ with respect to the grading by degree $d$. 
\end{thm}

\begin{proof} This is precisely the content of Theorem 4.2 in
\cite{CCRL}.
%(which includes relations (1)-(3)).  
The only adjustment is to think of the set of labels of external vertices in that theorem as partitioned into $m$ subsets (corresponding to  external vertices now lying on $m$ different segments rather than on only one). 
\end{proof}

\begin{rem}  Note that $\delta$ does not affect the order of a diagram.
\end{rem}

\begin{rem}  Cattaneo, Cotta-Ramusino, and Longoni define an algebra structure on their diagram complex via shuffle product of external vertices.  This product turns the complex into a differential graded Hopf algebra \cite[Theorem 3.2]{CCRL:Struct}.  The same can be done to our diagram complexes.  The shuffle product is simply defined on external vertices on each segment separately.  Thus if the sets of external vertices of $\Gamma_1$ and $\Gamma_2$ are $(V_{i_1}, V_{i_2}, ..., V_{i_m})$ and $(V_{j_1}, V_{j_2}, ..., V_{j_m})$ respectively, then a shuffle would be an $m$-tuple 
$$
(\sigma_1(V_{i_1}\cup V_{j_1}), \sigma_2(V_{i_2}\cup V_{j_2}), ..., \sigma_m(V_{i_m}\cup V_{j_m}))
$$
where each $\sigma_k$ is a permutation of the union  $V_{i_k}\cup V_{j_k}$ which preserves the linear order of elements of $V_{i_k}$ and $V_{j_k}$.
\end{rem}

%%%%%%%%%%%%%%%%%%%%%%%%%%%%%%%%%%%%%%%%%%%%%%%%%%%%%%%%%%%%%%%%%%%%%%%%

\section{Configuration space integrals}\label{S:B-TIntegrals}

%%%%%%%%%%%%%%%%%%%%%%%%%%%%%%%%%%%%%%%%%%%%%%%%%%%%%%%%%%%%%%%%%%%%%%%%

%%%%%%%%%%%%%%%%%%%%%%%%%%%%%%%%%%%%%%%%%%%%%%%%%%%%%%%%%%%%%%%%%%%%%%%%

\subsection{Bundles of compactified configuration spaces over links}\label{S:Compactification}

%%%%%%%%%%%%%%%%%%%%%%%%%%%%%%%%%%%%%%%%%%%%%%%%%%%%%%%%%%%%%%%%%%%%%%%%

Since we need to integrate over configuration spaces of points in $\R^n$, which are open and thus bring convergence of integrals into question, we will instead use their Fulton-MacPherson compactifications \cite{FM, AS}.  In these compactifications, configuration points are allowed to come together, while their directions and relative rates of approach are kept track of.  The resulting spaces are compact manifolds with corners which are homotopy equivalent to the open configuration spaces.  The original definition of the compactification replaces each diagonal in the product of copies of $\R^n$ by their blowups. We will use an alternative one which does not depend of blowups, due to Kontsevich and Soibelman \cite{KoSo} and Sinha \cite{S:Compact}.

Fix $n\geq 3$ and let $C_0(p, \R^n)$ denote the configuration space of $p$ points $x_1, x_2, ..., x_{p}$ in $\R^n$ (thus $x_i\neq x_j$ for all $i\neq j$).  
%Then identify $S^n$ with $\R^n\cup \{\infty\}$ and let $C(p, \R^n)$ be the fiber over $\infty$ of the projection
%\begin{align*}
%C_0(p+1, \R^n)  &  \longrightarrow  S^n  \\
%(x_1, x_2, ..., x_{p+1}) & \longmapsto x_{p+1}.
%\end{align*}
For points $x_i$, $x_j$, and $x_k$, with $1\leq i<j<k\leq p$, let 
$$
v_{ij}=\frac{x_j-x_i}{|x_j-x_i|}\in S^{n-1},\ \ \ \ \ a_{ijk}=\frac{|x_i-x_j|}{|x_i-x_k|}\in [0,\infty],
$$
where $[0,\infty]$ is the one-point compactification of $[0,\infty)$.  We then have a map
\begin{align}
\gamma\colon   C(p, \R^n) & \longrightarrow (\R^n)^p\times (S^{n-1})^{p \choose 2}  \times  [0,\infty]^{p  \choose 3}\label{E:CompactificationMap}  \\
             (x_1, ..., x_p) & \longmapsto (x_1, ..., x_p, v_{12}, ..., v_{ij}, ..., v_{(p-1)p}, a_{123}, ..., a_{ijk}, ..., a_{(p-2)(p-1)p}).\notag
\end{align}
\begin{defin}
Define $C[p, \R^n]$ to be the closure of $\gamma(C(p, \R^n))$ in $(\R^n)^p\times (S^{n-1})^{{p \choose 2}}\times  [0,\infty]^{p  \choose 3}$.  
\end{defin}
\begin{thm}[\cite{S:Compact}, Corollary 4.5 and Lemma 4.12]\label{T:DevTheorem} Space $C[p, \R^n]$ is a manifold with corners homotopy equivalent to the  Fulton-MacPherson compactification of $C(p, \R^n)$.
\end{thm}

In addition, $C[p, \R^n]$ is a  manifold with boundary components given by points colliding as well as points escaping to infinity. The latter faces come from the fact that the proper compactification (one that is actually compact) of $C(p, \R^n)$ is $C[\{x_1, ..., x_p\}\cup \{\infty\}, S^n]$, where $\{\infty\}$ is some fixed point on the sphere.  But since $\R^n$ is $S^n\setminus\{\infty\}$, we may as well consider $C[p, \R^n]$ while paying attention to extra faces occurring when configuration points tend to infinity.

 Codimension one faces of $C[p, \R^n]$  are given by a group of points coming together at the same time (as opposed to some colliding, then others joining them later).

\begin{defin}\label{D:LinksBundle}
Define the space $C[j_1, j_2,..., j_m, s;\, \Lk_{m}^n, \R^n]$ as the pullback 
$$
\xymatrix{
C[j_1, j_2,..., j_m, s;\, \Lk_{m}^n, \R^n]\ar[r]  \ar[d]  &
C[j_1+j_2 +\cdots+ j_m+ s,  \R^n]\ar[d]^{\pi}  \\
(C[j_1, \R]\times C[j_2, \R] \times\cdots\times C[j_m, \R])\times \Lk_{m}^{n}\ar[r]^-{ev} &
C[j_1+j_2+ \cdots+ j_m,  \R^n]
}
$$
where $\pi$ is the projection and $ev$ is the evaluation of a link with $m$ strands where the first strand is evaluated on the first $j_1$ points, second on next $j_2$ points, and so on.

Similarly define 
%$C[j_1, j_2,..., j_m, s;\, \HLk_{m}^n, \R^n]$ and 
$C[j_1, j_2,..., j_m, s;\, \Br_{m}^n, \R^n]$ by replacing $\Lk_{m}^n$ with 
%$\HLk_{m}^n$ and 
$\Br_{m}^n$ in the above diagram.
\end{defin}

Definition \ref{D:LinksBundle} 
%and \ref{D:HomotopyLinksBundle} 
is analogous to the one for the space of knots given on page 5283 of \cite{BT}.  Also analogous is the following result, which is a special case of Proposition A.3 in \cite{BT}. 

\begin{prop}\label{P:Bundles}
The pullbacks $C[j_1, j_2,..., j_m, s;\, \Lk_{m}^n, \R^n]$ and
%$C[j_1, j_2,..., j_m, s;\, \HLk_{m}^n, \R^n]$, and 
$C[j_1, j_2,..., j_m, s;\, \Br_{m}^n, \R^n]$
%, 
%and $C[j_1; j_2;...; j_m, s;\, \HLk_{m}^n, \R^n]$ 
are smooth manifolds with corners which fiber over $\Lk_{m}^{n}$
%, $\HLk_{m}^{n}$, 
and $\Br_{m}^{n}$, respectively.
\end{prop}

One should think of $C[j_1, j_2,..., j_m, s;\, \Lk_{m}^n, \R^n]$ (and analogously of the other space) as the space whose fiber over $LK\in \Lk_{m}^{n}$ is the space of $j_1+j_2 +\cdots+ j_m+ s$ configuration points in $\R^n$ with the first $j_1$ restricted to lie on the first strand of a given link, the second $j_2$ on the second, etc., while the last $s$ are free to move in $\R^n$.  We will denote such a fiber by $C[j_1, j_2,..., j_m, s;\, LK, \R^n]$ (or $C[j_1, j_2,..., j_m, s;\, B, \R^n]$ for $B\in \Br_{m}^n$).  The connection to the diagram complexes defined earlier should start becoming clearer; these configuration points correspond to vertices on or off the line segments.

\begin{rem}  As mentioned in the introduction, if would have been nice to define a diagram complex and configuration space integrals for the space of homotopy string links as well.  However, it is not clear how to establish \refP{Bundles} for this case.  The problem is that the homotopy link space is a subspace of the space of immersions, and Proposition A.3 in \cite{BT} requires points in the link space to be embeddings.
\end{rem}

%%%%%%%%%%%%%%%%%%%%%%%%%%%%%%%%%%%%%%%%%%%%%%%%%%%%%%%%%%%%%%%%%%%%%%%%

\subsection{Constructing cohomology classes of link spaces}\label{S:Forms}

%%%%%%%%%%%%%%%%%%%%%%%%%%%%%%%%%%%%%%%%%%%%%%%%%%%%%%%%%%%%%%%%%%%%%%%%%

We now want to construct morphisms of chain complexes $\LD^{o,*}\to \Omega^{(n-3)o+*}(\Lk_{m}^n)$
%, $\HLD^{o,d}\to \Omega^{*}(\HLk_{m}^n)$, 
and $\BD^{o,*}\to \Omega^{(n-3)o+*}(\Br_{m}^n)$, where $\Omega$ is the usual deRham algebra of differential forms.   The morphisms will be given by a modification of the Bott-Taubes configuration space integration \cite{BT}.  We will describe the construction for $\LD^{o,d}$ and then mention how it changes for the other space.

Let $n\geq 3$.  Given $\Gamma\in\LD^{o,d}$  with $j_i$ external vertices on the $i$th strand for $1\leq i\leq m$, $s$ internal vertices, and $|e|$ non-loop edges, consider the map
$$
\phi_{\Gamma}\colon C[j_1, j_2,..., j_m, s;\, \Lk_{m}^n, \R^n]
\longrightarrow S^{|e|(n-1)}
$$
given by the product of the normalized difference of those pairs of points in $C[j_1, j_2,..., j_m, s;\, \Lk_{m}^n, \R^n]$ for which there exists an edge in $\Gamma$.  (We first have to label the last $s$ points in $C[j_1, j_2,..., j_m, s;\, \Lk_{m}^n, \R^n]$ and corresponding internal vertices in $\Gamma$ the same way; this is already done for the external vertices by construction since those are labeled in a linear order as are the configuration points on the link strands.)

For each loop on $\Gamma$, instead of the normalized difference, we use the normalized derivative of the strand that has the loop.  We denote the product of all such maps by 
$$
(\partial\Lk_{m}^n)_{\Gamma} \colon C[j_1, j_2,..., j_m, s;\, \Lk_{m}^n, \R^n]
\longrightarrow S^{(\text{\# of loops in $\Gamma$})(n-1)}.
$$

Now let $sym_{S^{n-1}}$ be a normalized top form on $S^{n-1}$, by which we mean that its pullback via the antipodal map preserves it, up to sign.  We will further require $sym_{S^{n-1}}$ to be supported in $U_N$ and $U_S$ (see (2) of \refD{MainDef}).   Letting $\omega$ be the product of $|e|+(\text{\# of loops in $\Gamma$})$ such top forms on  $S^{n-1}$, define the pullback form
$\alpha$ on $C[j_1, j_2,..., j_m, s;\, \Lk_{m}^n, \R^n]$ by
$$
\alpha=(\phi_{\Gamma}\times (\partial\Lk_{m}^n)_{\Gamma})^*\omega.
$$
  Finally, $\alpha$ can be pushed forward along the fiber of the bundle map 
$$
\pi_{\Lk}\colon C[j_1, j_2,..., j_m, s;\, \Lk_{m}^n, \R^n]\longrightarrow \Lk_{m}^{n}
$$

We will denote the resulting form by $(I_{\Lk})_\Gamma$.  Its degree is  
\begin{align*}
\degree (I_{\Lk})_\Gamma & =(\text{degree of $\alpha$})-(\text{dimension of fiber of $\pi_{\Lk}$}) \\
  & =(n-1)|e|-(ns+j_1+j_2+\cdots+j_m)  \\
  & =(n-1)|e|-n|v_{int}|-|v_{ext}|  \\
  & = (n-3)o+d
\end{align*}
where as usual $o=\ord\Gamma$ and $d=\degree\Gamma$.

The value of $(I_{\Lk})_\Gamma$ on an $((n-3)o+d)$-chain sitting over $LK\in \Lk_{m}^n$ is therefore
$$
(I_{\Lk})_\Gamma(LK)=\int\limits_{\pi_{\Lk}^{-1}(LK)=C[j_1, j_2,..., j_m, s;\, LK, \R^n]}\alpha.
$$

We thus get a linear map
\begin{align}
I_{\Lk}\colon \LD^{o,d} & \longrightarrow \Omega^{(n-3)o+d}(\Lk_{m}^{n}) \label{E:LinksMap}\\
\Gamma  & \longmapsto  \Big(LK\mapsto (I_{\Lk})_{\Gamma}(LK)=\int\limits_{C[j_1, j_2,..., j_m, s;\, LK, \R^n]}\alpha\Big).\notag
\end{align}

The corresponding construction for 
%$\Gamma\in\HLD^{o,d}$ or 
$\Gamma\in\BD^{o,d}$ is identical, except now the total space is 
%$C[j_1, j_2,..., j_m, s;\, \HLk_{m}^n, \R^n]$ or 
$C[j_1, j_2,..., j_m, s;\, \Br_{m}^n, \R^n]$.  However, since  $\BD^{o,d}$ has no loops, $\alpha$ is defined just as the pullback $\phi_{\Gamma}^*\omega$.  The map corresponding to $I_{\Lk}$  will be denoted by 
%$I_{\HLk}$ and 
$I_{\Br}$.

\begin{prop}\label{P:MapsRespectRelations}
The maps
\begin{align*}
I_{\Lk}\colon\LD^{o,d} & \longrightarrow \Omega^{(n-3)o+d}(\Lk_{m}^{n}) \\
%I_{\HLk}\colon\HLD^{o,d} & \longrightarrow \Omega^{(n-3)o+d}(\HLk_{m}^{n}) \\
I_{\Br}\colon\BD^{o,d} & \longrightarrow \Omega^{(n-3)o+d}(\Br_{m}^{n}) 
\end{align*} 
are compatible with the relations from \refD{DiagramSpaces}. 
\end{prop}

\begin{proof}
%Use lemma 8.6 and stuff surrounding it in the formality paper.  Sakai's Prop 3.1 is also similar.
To simplify notation, denote both maps by $I$.

For the first relation, if a diagram has a double edge, then two of the maps to the product of spheres are the same, and $I$ thus factors through $S^{(n-1)|e|(\text{\# of loops in $\Gamma$})-1}$.  For dimensional reasons, the pullback of $\omega$ thus must be 0 (see, for example,  \cite[Proposition 5.24]{HLTV} for details).

It is also clear that $I$ respects the orientation relations (2) and (3) since 
permuting the labels of the configuration points off the link (these correspond to internal vertices) may change the orientation of the fibers $\pi_{\Lk}^{-1}(LK)$ for $n$ odd;
changing edge orientations composes $\phi_{\Gamma}$ with some number of antipodal maps which may introduce a sign for $n$ odd;
%switching the ordering of half-edges of a loop ...\footnote{I don't know why this is here.  See first bullet on page 971 of Catt et al.  Can I get rid of ordering of half-edges?  Maybe for them it has to do with the fact that they're working on the circle, since same knot can be parametrized by $t$ and $1-t$, but that can't happen for us?};
and permuting the edge labels permutes the various wedge factors $sym_{S^{n-1}}$ in $\omega$ and may introduce a sign for $n$ even.
All the signs in the diagram complexes have been defined to be compatible with the corresponding sign changes in $I$.
\end{proof}

%
%Now let $\delta_{S^{n-1}}$ be the bump form on $S^{n-1}$ concentrated at the north pole whose integral over the sphere is 1.

%
%Problem is that I don't know how the maps are related when I switch the form.  The chain map could be something totally different.

The following is the main theorem of the paper.

\begin{thm}\label{T:MainThm}  Let $m\geq 1$. For $n> 3$, the maps $I_{\Lk}$ and $I_{\Br}$
%\begin{align*}
%I_{\Lk}\colon\LD^{o,d} & \longrightarrow \Omega^{(n-3)o+d}(\Lk_{m}^{n}) \\
%%I_{\HLk}\colon\HLD^{o,d} & \longrightarrow \Omega^{(n-3)o+d}(\HLk_{m}^{n}) \\
%I_{\Br}\colon\BD^{o,d} & \longrightarrow \Omega^{(n-3)o+d}(\Br_{m}^{n}).
%\end{align*}
are morphisms of cochain complexes.
\end{thm}

\begin{rem}  This is a generalization of the same statement for knots, namely Theorem 1.1 in \cite{CCRL}.
\end{rem}

\begin{proof}

Consider first the map $I_{\Lk}$.  
The differential in $\Omega^{(n-3)o+d}(\Lk_{m}^{n})$ is given by Stokes' Theorem which says that 
$$
d(I_{\Lk})_\Gamma=\pi_*d\alpha-\int\limits_{\partial C[j_1, j_2,..., j_m, s;\, LK, \R^n]}\alpha.
$$
The first term is zero since $\alpha$ is closed.  The second is the sum of the pushforwards of $\alpha$ along all the codimension one boundaries of $C[j_1, j_2,..., j_m, s;\, LK, \R^n]$ (to which $\alpha$ extends smoothly; see \cite{BT, V:SBT}).

We identify three types of codimension one boundary:
\begin{itemize}
\item  \emph{principal faces}, characterized by two points coming together;
\item  \emph{hidden faces}, characterized by more than two points coming together;
%\item  \emph{anomalous faces}, characterized by all points coming together. 
\item  \emph{faces at infinity}, characterized by one or more points tending to infinity.
\end{itemize}
The differential in $\LD^{o,d}$ has been defined so as to correspond precisely to certain principal faces.  Namely, for  $\Gamma\in \LD^{o,d}$, $(I_{\Lk})_{d\Gamma}$ is the form obtained by adding the integrals of $\alpha$ over all principal faces of $C[j_1, j_2,..., j_m, s;\, \Lk_{m}^n, \R^n]$.  We just need to show that the integral along those principal faces which do not share a mixed or internal edge is zero.  This is accomplished by noting that, if $x_1$ and $x_2$ collide and there is no map keeping track of the direction between them, then the restriction of the map $\phi_{\Gamma}\times (\partial\Lk_{m}^n)_{\Gamma}$ to that principal face, which we will denote by $\partial_{x_1=x_2}C[j_1, j_2,..., j_m, s;\, LK, \R^n]$, factors through a space in which $x_1$ and $x_2$ are allowed to pass through each other.  In other words, let $C'[j_1, j_2,..., j_m, s;\, \Lk_{m}^n, \R^n]$ be the space where the $x_1=x_2$ diagonal has not been blown up (or map $\gamma$ has been modified appropriately).  Then we have a factorization
$$
\xymatrix{
\partial_{x_1=x_2}C[j_1, j_2,..., j_m, s;\, LK, \R^n] \ar[rr]^-{(\phi_{\Gamma}\times (\partial\Lk_{m}^n)_{\Gamma})|_{x_1=x_2}}\ar[dr] & & S^{(|e|+(\text{\# of loops in $\Gamma$}))(n-1)}  \\
&  \partial_{x_1=x_2}C'[j_1, j_2,..., j_m, s;\, LK, \R^n]  \ar[ur] &
}
$$
The dimension of $\partial_{x_1=x_2}C'[j_1, j_2,..., j_m, s;\, LK, \R^n]$ is strictly less than that of $\partial_{x_1=x_2}C[j_1, j_2,..., j_m, s;\, LK, \R^n]$ since $n> 3$ ($x_1$ and $x_2$ coming together in the first space is a boundary of codimension at least 3).  Thus $\alpha$ must be zero.

Note that this argument works any time the subset of colliding vertices can be broken up into two connected subsets.  Details in the case of knots and $n=3$ can be found in \cite[Proposition 4.1]{V:SBT}, but the statement generalizes easily to our case.

To prove the theorem, it then remains to show that the integrals along hidden faces and  faces at infinity vanish.

\vskip 4pt
\emph{Vanishing along hidden faces:}  The proof for the case of knots \cite[Theorem A.6]{CCRL}  applies in an identical way here.  It is important to note that this is where the assumption $n>3$ is needed (rest of the arguments work for $n=3$ as well).
%For $n=3$, the proof is given in  \cite[Section 4.3]{V:SBT} and can also be repeated verbatim, with a slight modification:  The key lemmas for both $n>3$ and $n=3$ are due to Kontsevich \cite{K:Vass}.   In \cite[Theorem A.6]{CCRL}, they are stated for any diagrams while in  in \cite{V:SBT} are only stated for trivalent diagrams, they
\vskip 4pt
\emph{Vanishing along faces at infinity:}  This is essentially the content of \cite[Lemma 3.9]{V:IT}, which is the same statement for knots.  The only new kind of a face at infinity is the one where two external vertices on different strands escape to infinity.  But then the vector between them is constant in the limit (since our links are linear outside a compact set) and so the restriction of $\phi_{\Gamma}\times (\partial\Lk_{m}^n)_{\Gamma}$ to this face factors through a point.  The pullback $\alpha$ again must be zero.
%\vskip 4pt
%\emph{Vanishing along anomalous face:}  First observe that this case only occurs when all the external vertices are on a single strand, and therefore cannot happen for $\HLk_{m}^n$ or $\Br_{m}^{n}$.

The arguments for $I_{\Br}$ are identical, except the last relation from \refD{DiagramSpaces} has to be taken into account.  But this is immediate since the condition (2) of \refD{MainDef} means that, if $x_1$ and $x_2$ are points on the same braid strand, then the vector $(x_1-x_2)/|x_1-x_2|$ is never in $U_N$ or $U_S$, and therefore the integral of the pullback of the normalized top form concentrated in those neighborhoods via this map must be zero.  One has such an integral precisely when there is a chord on a single line segment in a diagram in $\BD^{o,d}$, in which case the diagram is set to zero.
\end{proof}

\begin{rem}
Changing the form $sym_{S^{n-1}}$ to another symmetric form with support in $U_N$ and $U_S$ does not change the resulting cohomology class as the difference of integrals along the two forms for any diagram $\Gamma$ is exact \cite[Proposition 4.5]{CCRL}.
\end{rem}

\vskip 6pt

It should be clear that the maps $I_{\Lk}$ and $I_{\Lk}$ are compatible with the usual inclusions of long knots and braids into links.  Namely, with $\K^n$ as before denoting the space of long knots in $\R^n$, one has a map
$$
\K^n\longrightarrow \Lk_m^n
$$
given by replacing the first strand of the long unlink by a long knot (scaled sufficiently so that it is away from the rest of the strands), and the induced quotient map on deRham algebras 
$$
\Omega^*(\Lk_m^n)\longrightarrow \Omega^*(\K_n).
$$
Associated to $\K^n$, one also has the diagram complex $\mathcal{D}$ of Cattaneo, Cotta-Ramusino, and Longoni \cite{CCRL} mentioned in the introduction, which is defined exactly as $\LD$ except that it consists of diagrams with only one line segment which contains all external vertices.  There is thus an inclusion 
$$
\mathcal{D}\longrightarrow \LD
$$
given by replacing the first segment of the empty diagram by a given diagram in $\mathcal{D}$.  Further, the integration map in \cite{CCRL} 
$$
I_{\K}\colon \mathcal{D} \longrightarrow \Omega^*(\K_n).
$$
is defined exactly as ours.  Putting this together, we get that the composed map
$$
\mathcal{D}\longrightarrow \LD\longrightarrow \Omega^*(\Lk_m^n)  \longrightarrow \Omega^*(\K_n)
$$   
is precisely $I_{\K}$.

The situation is exactly the same for the inclusion of braids into links, i.e.~the classes defined by $I_{\Lk}$ and $I_{\Br}$ map to each other in a natural way.

%%%%%%%%%%%%%%%%%%%%%%%%%%%%%%%%%%%%%%%%%%%%%%%%%%%%%%%%%%%%%%%%%%%%%%%%%%%%%%%%

\section{The case $n=3$}\label{S:Casen=3}

%%%%%%%%%%%%%%%%%%%%%%%%%%%%%%%%%%%%%%%%%%%%%%%%%%%%%%%%%%%%%%%%%%%%%%%%%%%%%%%%

%We wish to single out the case $n=3$ for two reasons:
%\begin{itemize}
%\item In the case of $\Lk_{m}^{3}$, the integral $I_{\Lk}$ along the anomalous face does not vanish and a correction term has to be introduced.  This is discussed in section \ref{S:AnomalousFaces};
%\item 
The dimension $n=3$ is of course in many respects the most interesting one, and we wish to pay attention to what classes in degree zero, or link invariants, our configuration space integrals might produce.  As it turns out, they give all finite type, or Vassiliev, invariants.  This is the subject of section \ref{S:UniversalFiniteType} with some preliminary results about the diagram complexes and certain faces of compactified configuration spaces associated to this case established in sections \ref{S:TrivalentDiagrams} and \ref{S:AnomalousFaces}.
%\end{itemize}

%%%%%%%%%%%%%%%%%%%%%%%%%%%%%%%%%%%%%%%%%%%%%%%%%%%%%%%%%%%%%%%%%%%%%%%%%%%%%%%%

\subsection{Trivalent diagrams}\label{S:TrivalentDiagrams}

%%%%%%%%%%%%%%%%%%%%%%%%%%%%%%%%%%%%%%%%%%%%%%%%%%%%%%%%%%%%%%%%%%%%%%%%%%%%%%%%

%To simplify notation,  in this section we will refer to each of the diagram spaces $\LD^{o,d}$ and $\BD^{o,d}$ simply by  $\mathcal{D}^{o,d}$ and make the distinction when necessary.

To see what link invariants, i.e.~elements of $H^0(\Lk_{m}^{3})$  and $H^0(\Br_{m}^{3})$, our integrals generate, first note that if $n=3$, then to obtain classes in $H^0$, we must set $d=0$.  It is a simple combinatorial exercise to see that this forces diagrams in both of our diagram complexes to have trivalent internal and univalent external vertices (where we do not count arcs emanating from external vertices into their valence, only edges).  Note also that such a graph cannot have loops.  Its order is then half the total number of vertices. 
Such diagrams are known as \emph{trivalent diagrams}.  Before we state the main results about them which are of interest here, it is convenient to 
identify $H^0(\mathcal{LD}^{o,0})$ and $H^0(\mathcal{BD}^{o,0})$ with their duals (this is fine since we have a basis, namely the diagrams), in order to connect to the theory of finite type invariants in section \ref{S:UniversalFiniteType}.  We then have

\begin{thm}\label{T:TrivalentCohomology}
The cohomology group $H^0(\mathcal{LD}^{o,0})$ is isomorphic to the dual of the subspace of trivalent diagrams $\mathcal{LD}^{o,0}$ in $\mathcal{LD}^{o,d}$ modulo \emph{STU} and \emph{1T} relations pictured in Figures \ref{Fig:STURelation} and \ref{Fig:1TRelation}.

The result is the same for  $H^0(\mathcal{BD}^{o,0})$ except the 1T relation is not needed.
\end{thm}

%%\raisebox{-.15cm}{
%\begin{figure}
%  \centering
%  \fbox{\includegraphics[height=30mm]{STURelation.jpg}}
%  \caption{$STU$ relation.  No other edges connect to the pictured vertices and the diagrams are same outside pictured portions.}
%  \label{Fig:STURelation}
%\end{figure}
%%}

\begin{figure}[h]
\input{sturelation.pstex_t}
\caption{$STU$ relation.  No other edges connect to the pictured vertices and the diagrams are same outside pictured portions.}
\label{Fig:STURelation}
\end{figure}

%%\raisebox{-.15cm}{
%\begin{figure}
%  \centering
%  \fbox{\includegraphics[height=14mm]{1TRelation.jpg}}
%  \caption{$1T$ relation.   The rest of the diagram is entirely outside the pictured portion.}
%  \label{Fig:1TRelation}
%\end{figure}
%%}

\begin{figure}[h]
\input{1trelation.pstex_t}
\caption{$1T$ relation.   The rest of the diagram is entirely outside the pictured portion.}
\label{Fig:1TRelation}
\end{figure}

\begin{rem}
For the braid case, 1T relation is unnecessary since those diagrams do not have chords with endpoints on the same segment.
\end{rem}

\begin{proof} The proof is immediate from considering what the adjoint to the differential $\delta$ must be.  Since $\delta$ identifies vertices, its adjoint ``blows them up" in all possible ways.  Thus blowing up an external vertex with two edges emanating from it gives the STU relation and blowing up a vertex with a loop gives the 1T relation.  It turns out one need not consider the blowups of  internal vertices, since the relation in that case (so-called \emph{IHX} relation) follows from the STU relation \cite[Theorem 6]{BN:Vass}.  More details in the case of knots can be found in \cite[Section 3]{Long:Classes}.
\end{proof}
%
%Elements of $H^0(\mathcal{D}^{k,0})$ are often called \emph{weight systems of degree $o$} and denoted by $\W_o$.

%%%%%%%%%%%%%%%%%%%%%%%%%%%%%%%%%%%%%%%%%%%%%%%%%%%%%%%%%%%%%%%%%%%%%%%%%%%%%%%%

\subsection{Anomalous faces}\label{S:AnomalousFaces}

%%%%%%%%%%%%%%%%%%%%%%%%%%%%%%%%%%%%%%%%%%%%%%%%%%%%%%%%%%%%%%%%%%%%%%%%%%%%%%%%

Using maps $I_{\Lk}$ and $I_{\Br}$, we want to construct maps
\begin{align*}
I_{\Lk}^*\colon & H^0(\LD^{k,0})\longrightarrow H^0(\Lk_{m}^{3}) \\
I_{\Br}^*\colon & H^0(\BD^{k,0})\longrightarrow H^0(\Br_{m}^{3}).
\end{align*}
%given by
%\begin{equation}\label{E:n=3Map}
%I^*(W)(LK)=\sum_{\Gamma\in \mathcal{D}^{k,o}} W(\Gamma)I_{\Gamma}(LK)
%\end{equation}
%where $W\in\W_k$.

However, the proof that the integrals along hidden faces vanish for trivalent diagrams and $n=3$ breaks down for the special face when all the points in a single (or more) \emph{isolated strand component} of $C[j_1, j_2,..., j_m, s;\, LK, \R^n]$ or $C[j_1, j_2,..., j_m, s;\, B, \R^n]$ come together.  Such a face is called \emph{anomalous}.  By such a component we mean a subset of vertices corresponding to a connected component of a diagram $\Gamma$ all of whose external vertices are on a single strand, and the arcs between those external vertices do not contain external vertices of any other components.
%This of course cannot happen in  $\HLk_{m}^{3}$ or $\Br_{m}^{3}$ since all diagrams associated to these spaces must have external vertices on more than one segment and there will hence be configuration points on more than one strand of a homotopy link or a braid.
In other words, an anomalous face has a subset of the configuration points colliding, but there is no map measuring the direction between a point in this subset and a point outside the subset.  It should be clear why external vertices are required to lie on a single segment; configuration points on link components can only collide if they lie on the same strand.

Notice that if $\Gamma$ has such a component $\Gamma_{str}$, then by Fubini's Theorem we have
$$
I_{\Gamma}=I_{\Gamma_{str}}I_{\Gamma\setminus\Gamma_{str}},
$$
where $\Gamma\setminus\Gamma_{str}$ denotes $\Gamma$  with all the vertices and edges in $\Gamma_{str}$ removed.  We can thus without loss of generality assume from now on that a diagram giving rise to an anomalous face has one connected component all of whose external vertices are on a single line segment.  This is why anomalous faces can be regarded as a knotting, rather than linking, phenomenon.

To fix the contribution of the anomalous faces, one produces, for each $\Gamma$, a correction term.   By above remarks, the computation of this term reduces to its computation in the case of knots.  This was done in \cite{Thurs} (see also \cite{V:SBT}).  Adapting the notation to the case of links and braids, the reformulation of those results is as follows.  It is important to note that one now needs to use a rotationally invariant volume form on $S^2$ rather than one concentrated around the poles.

\begin{prop}\label{P:Anomaly}  For $n=3$, 
\begin{itemize}
\item  if $\Gamma_{str}\in \LD^{k,0}$ contains a chord, then the pushforward of the restriction of $I_{\Lk}$ to the anomalous face vanishes \cite[Corollary 4.3]{V:SBT}.  (Recall that diagams in $\BD^{k,0}$ do not contain chords on single strands.)
\item  if $\Gamma_{str}\in \LD^{k,0}$ does not contain a chord or if $\Gamma_{str}\in \BD^{k,0}$, then the pushforward of the restriction of $I_{\Lk}$ or $I_{\Br}$ to the anomalous face equals 
\begin{equation}\label{E:Anomalous}
\mu_{\Gamma_{str}}\int\limits_{C[1,0; K, \R^3]\cong C(1,K)\cong K}\left(\frac{K'}{|K'|}\right)^* vol_{S^{2}}
\end{equation}
where
\begin{itemize}
\item $\mu_{\Gamma_{str}}$ is a real number which depends on $\Gamma_{str}$;
\item $K$ is the strand corresponding to the segment with all the external vertices in $\Gamma_{str}$ and $K'$ is its derivative;
%\item $x_1$ and $x_2$ are points on $\K$;
\item $vol_{S^{2}}$ is the rotationally invariant normalized top form on $S^{2}$
\end{itemize}
\cite[Proposition 4.8]{V:SBT}
\end{itemize}

\end{prop}

By discussion above, the correction term in the case of a diagram which was not necessarily concentrated on a single segment would, again by Fubini's Theorem, be a product of correction terms for each such component of the diagram.  The integral over the face which has configuration points from two or more such components colliding does not have to be taken into account; any time a subset of the colliding vertices can be broken up into two connected subsets, the integral is zero (see proof of \refT{MainThm}).

\begin{rem}  Note that, in case of braids, $K$ is necessarily the unknot.  Thus understanding the anomaly in the case of braids amounts to understanding configuration space integrals for the unknot.
Also note that, if we could have used a bump form as in the previous section to establish this proposition, the anomalous face would not be an issue for braids.  Namely, since the derivative along each component of the braid has positive first component, the integral from equation \eqref{E:Anomalous} would be zero since the bump form is concentrated at the poles.
\end{rem}

Now recall the definition of the maps $I_{\Lk}$ and $I_{\Br}$ from equation \eqref{E:LinksMap}.  Using \refP{Anomaly}, we get 

\begin{thm}\label{T:n=3CochainMap}
The restriction of the map
$$
\overline{I}_{\Lk}\colon \LD^{k,0}\longrightarrow \Omega^0(\Lk_{m}^{3})
$$
given by 
$$
\Gamma\longmapsto 
\begin{cases}
I_{\Lk}(\Gamma)- \mu_{\Gamma}
%\left(
\int\limits_{C[2,0; K, \R^3]\simeq C[2,K]}\left(\frac{x_1-x_2}{|x_1-x_2|}\right)^* vol_{S^{n-1}}
%\right)^{\text{\# loops on $\Gamma$}+1}
, &  \Gamma=\Gamma_{str};  \\
I_{\Lk}(\Gamma),   & \text{otherwise}
\end{cases}
$$
to the anomalous face of each diagram $\Gamma\in\LD^{k,0}$ is zero.  Here $\Gamma_{str}$ has all its external vertices on one strand and it has no chords.  As before, $K$ is the strand of $LK$ corresponding to the segment on which $\Gamma_{str}$ is concentrated.

Same statement is true for the map
$$
\overline{I}_{\Br}\colon \BD^{k,0}\longrightarrow \Omega^0(\Br_{m}^{3})
$$
defined as above but with $I_{\Lk}$ replaced by $I_{\Br}$ and $\Gamma\in\BD^{k,0}$.
\end{thm}

\begin{proof}  We give the proof for the case of $\overline{I}_{\Lk}$.  The argument is the same for $\overline{I}_{\Br}$.

%First note that all the vanishing arguments from the proof of \refT{MainThm} for faces other than the anomalous ones hold for $\Lk_{m}^{3}$.  It therefore remains to address the anomalous face.

If $\Gamma$ is not concentrated on one strand (i.e.~not all its external vertices are on one strand), then it has no anomalous face and \refT{MainThm} shows that the pushforward along all non-principal faces vanishes.   

If $\Gamma$ is concentrated on one strand and it has a chord, then the pushforward along the anomalous face vanishes by \refP{Anomaly}.

%If $\Gamma$ is concentrated on one strand and it contains more than one loop, then the restriction of $\overline{I}_{\LK}$ to the anomalous face factors through a space with dimension lower than $S^{(n-1)|e|(\text{\# of loops in $\Gamma$})-1}$ and the pullback of $\omega$ thus must be 0 as in the proof of \refP{MapsRespectRelations} (i.e.~after all vertices are collapsed, the resulting diagram has a vertex with more than one loop; this is the case of a diagram with a double edge).

%If $\Gamma$ is concentrated on one strand and it contains one loop,

If $\Gamma$ is concentrated on one strand and it does not have a chord, to find the correction to $I_{\Lk}$, by \refP{Anomaly} we need a space whose boundary is $C[1,0; \K, \R^3]$ and a map which becomes the derivative on that boundary.  But this space is simply the configuration space of two points on $K$ and the map is the normalized difference of those points.  This is precisely the integral 
\begin{equation}\label{E:Fix}
\int\limits_{C[2,0; K, \R^3]\simeq C[2,K]}\left(\frac{x_1-x_2}{|x_1-x_2|}\right)^* vol_{S^{n-1}}
\end{equation}
from the statement of the theorem.  The correction coefficient $\mu_{\Gamma}$ comes from \refP{Anomaly}.
%To see what happens in the case when $\Gamma$ also has loops, observe that the tangential maps in the product $(\partial\Lk_{m}^{n})_{\Gamma}$ are independent from $\phi_{\Gamma}$.  This means that, given $LK\in\Lk_{m}^{n}$, we have by Fubini's Theorem
%\begin{align*}
%\int\limits_{C[j_1, j_2,..., j_m, s;\, LK, \R^n]}  \alpha = & 
%\left(\prod_{i=1}^{m}\left(\int\limits_{K_i} \left(\frac{K'_i}{|K'_i|} \right)^*vol_{S^{n-1}}    \right)^{\text{\# loops on $\Gamma$'s $i$th segment}}\right) \\
%  &  \cdot\int\limits_{C[j_1, j_2,..., j_m, s;\, LK, \R^n]}  \phi_{\Gamma}^*(vol_{S^{n-1}})^{|e|}
%\end{align*}
%where $K_i$ is the $i$th strand of $LK$ and $|e|$ is the number of non-loop edges of $\Gamma$.  Now notice that the tangential maps $K'_i/|K'_i|$ remain the same after passing to any of the faces of $C[j_1, j_2,..., j_m, s;\, LK, \R^n]$. So in the case when $\Gamma$ is concentrated on the $i$th strand and has no chords, denoting by $A$ the anomalous face of $C[j_1, j_2,..., j_m, s;\, LK, \R^n]$ we get by \refP{Anomaly} that
%\begin{align*}
%\int\limits_{A}\alpha & =  \left(\int\limits_{K_i} \left(\frac{K'_i}{|K'_i|} \right)^*vol_{S^{n-1}}    \right)^{\text{\# loops on $\Gamma$}}
%\cdot
%\mu_{\Gamma}\int\limits_{K_i} \left(\frac{K'_i}{|K'_i|} \right)^*vol_{S^{n-1}} \\
%&  =\mu_{\Gamma}\left(\int\limits_{K_i} \left(\frac{K'_i}{|K'_i|} \right)^*vol_{S^{n-1}}    \right)^{\text{\# loops on $\Gamma$}+1}.   
%\end{align*}
%Looking back at the discussion preceding equation \eqref{E:Fix}, this is therefore precisely what we wanted to show.
\end{proof}

\begin{rem}
It is an open question whether $\mu_{\Gamma}=0$, but some computations done by D. Thurston in low degrees in the knot case suggest that it is.
\end{rem}

%THIS IS AN ISSUE:

%\begin{rem}
%Notice that in this section we switched the choice of our symmetric form on $S^{n-1}$ from a bump for to the $SO(n)$-invariant one.  Working with this form is necessary for the result from second part of \refP{Anomaly}.
%\end{rem}

We now have 

\begin{thm}\label{T:MapsForn=3}  Let $m\geq 1$.  There exist maps 
\begin{align*}
\overline{I}_{\Lk}^*\colon H^0(\mathcal{\LD}^{k,0}) & \longrightarrow H^0(\Lk_{m}^{3})  \\
W  & \longmapsto   \left( LK\longmapsto \sum_{\Gamma\in \LD^{k,o}} W(\Gamma)(\overline{I}_{\Lk})_{\Gamma}(LK)  \right)
\end{align*}
and 
\begin{align*}
\overline{I}_{\Br}^*\colon H^0(\mathcal{\BD}^{k,0}) & \longrightarrow H^0(\Br_{m}^{3})  \\
W  & \longmapsto   \left( B\longmapsto \sum_{\Gamma\in \BD^{k,o}} W(\Gamma)(\overline{I}_{\Br})_{\Gamma}(B)  \right).
\end{align*}
%
%\overline{I}_{\Br}^*\colon H^0(\mathcal{\BD}^{k,0}) & \longrightarrow H^0(\Br_{m}^{3})  \\
%\end{align*}
%given by
%\begin{equation}\label{E:n=3Map}
%\overline{I}_{\Lk}^*(W)(LK)=\sum_{\Gamma\in \mathcal{D}^{k,o}} W(\Gamma)(\overline{I}_{\Lk})_{\Gamma}(LK)
%\end{equation}
%where 
%\begin{itemize}
%\item $W\in H^0(\mathcal{\LD}^{k,0})$ 
%%$H^0(\mathcal{\HLD}^{k,0})$, 
%or  $H^0(\mathcal{\BD}^{k,0})$,
%\item $I_{\Gamma}$ is the map from equation \eqref{E:LinksMap} (and $I$ is either $\overline{I}_{\Lk}$ 
%%$I_{\HLk}$, 
%or $I_{\Br}$), and 
%\item $\mathcal{D}^{k,o}=\mathcal{\LD}^{k,0}$
%%, $\mathcal{\HLD}^{k,0}$, 
%or $\mathcal{\BD}^{k,0}$, 
%\end{itemize}
%respectively.
\end{thm}

\begin{proof}
We need to show that the integrals along all faces of $\Gamma$ vanish or cancel out.  

For principal faces where points which do not share an edge collide, comments from the proof of \refT{MainThm} apply.  For those that do, there are always three diagrams which differ either as in $STU$ (or the $IHX$) relations.  The integrals along the faces where each of the two pictured vertices in $STU$ come together are the same.  The signs of the integrals match the signs in the relation.  Thus these contributions cancel in the sum.  More details about this in the case of knots can be found in \cite[Section 4.4]{V:SBT}.  In fact, the arguments are identical in the case of links.

For non-anomalous hidden faces, the argument depends on enumerating the possibilities for the  vertex which has an edge that is not involved in the collision (and there is one since our diagrams are connected and not all points are colliding).  The fact that the valence of any vertex is at most three is crucial for this to work.  Details can be found in \cite[Proposition 4.4]{V:SBT}.

For faces at infinity, the same argument as in \refT{MainThm} works.

Finally, the anomalous face is taken care of by definitions of $\overline{I}_{\Lk}$ and $\overline{I}_{\Br}$.
\end{proof}

%%%%%%%%%%%%%%%%%%%%%%%%%%%%%%%%%%%%%%%%%%%%%%%%%%%%%%%%%%%%%%%%%%%%%%%%%%%%%%%%

\subsection{Finite type invariants of links and braids}\label{S:UniversalFiniteType}

%%%%%%%%%%%%%%%%%%%%%%%%%%%%%%%%%%%%%%%%%%%%%%%%%%%%%%%%%%%%%%%%%%%%%%%%%%%%%%%%

%

%
%%%%%%%%%%%%%%%%%%%%%%%%%%%%%%%%%%%%%%%%%%%%%%%%%%%%%%%%%%%%%%%%%%%%%%%%%%%%%%%%%

%\subsection{Review of finite type theory}\label{S:FiniteTypeReview}

%
%%%%%%%%%%%%%%%%%%%%%%%%%%%%%%%%%%%%%%%%%%%%%%%%%%%%%%%%%%%%%%%%%%%%%%%%%%%%%%%%%

%
Finite type invariants have received much attention in recent years because of the many connections they have to physics, 3-manifold theory, etc.  Their study was initiated by Vassiliev \cite{Vass:Cohom} who came across them by studying embeddings as complements of immersions.  Kontsevich \cite{K:Fey} exhibited an isomorphism between finite type invariants of knots and the dual of trivalent diagrams.  This is now known as the fundamental theorem of finite type theory.
 An alternative proof of this theorem uses configuration space integrals \cite{Thurs} (also see \cite{V:SBT} for details).  \refT{UniversalFiniteType} below is a direct generalization of that statement to spaces of links.

%
%To simplify notation, let $\calX$ stand for either of the spaces $\Lk_{m}^{3}$ or $\Br_{m}^{3}$.  We will as usual refer to elements of either of these spaces as ``links".   Also let $\mathcal{D}^{k,o}$ denite either of the trivalent diagram subspaces of  $\LD^{k,0}$ and $\BD^{k,0}$ and let $\overline{I}$ stand for either of the chain maps $\overline{I}_{\Lk}$ and $\overline{I}_{\Br}$.  

The goal of this section is to show that the elements of $H^0(\Lk_{m}^{3})$ and $H^0(\Br_{m}^{3})$ which arise via the maps $\overline{I}_{\Lk}$ and $\overline{I}_{\Br}$ from \refT{MapsForn=3} are precisely \emph{finite type link invariants}.  
We first briefly review some definitions and results from finite type theory.

Define a \emph{$k$-singular link} to be a link as usual except for a finite number of double points where the derivatives are independent.  For $\Lk_{m}^{3}$, the strands involved in a singularity can come from a single strand or two different strands, while for $\Br_{m}^{3}$, they necessarily come from different strands (as braids always ``move in the same direction").

Given a link invariant $V\in H^0(\Lk_{m}^{3})$ or $V\in H^0(\Br_{m}^{3})$, extend it to $k$-singular links via repeated use of the
\emph{Vassiliev skein relation} from Figure \ref{Fig:SkeinRelation}.
%%\vspace{-11pt}
%\begin{figure}[h]\label{Fig:SkeinRelation}
%\begin{center}
%\input{Vassiliev.pstex_t}
%\caption{Vassiliev skein relation}
%\end{center}
%\end{figure}
%%\vspace{-11pt}

%%\raisebox{-.15cm}{
%\begin{figure}
%  \centering
%  \fbox{\includegraphics[height=14mm]{SkeinRelation.jpg}}
%  \caption{Skein relation.}
%  \label{Fig:SkeinRelation}
%\end{figure}
%%}

\begin{figure}[h]
\input{skeinrelation.pstex_t}
\caption{Skein relation.}
\label{Fig:SkeinRelation}
\end{figure}

The pictures represent a neighborhood of a singularity, and outside of it, the three links are identical.  The two pictures on the right  are called \emph{resolutions} of the singularity.
The order in which the singularities are resolved  does not matter because of the sign conventions.

\begin{defin}\label{D:FiniteType}
Link invariant $V$ is a \emph{ finite type $k$ invariant} (or \emph{Vassiliev of type $k$}) if it vanishes on links with $k+1$ double points. 
\end{defin}

Let $\mathcal{LV}_{k}$ and $\mathcal{BV}_{k}$ be the collections
of all type $k$ invariants for the link space and the braid space, respectively, and note that
$\mathcal{LV}_{k-1}\!\subset\!\mathcal{LV}_{k}$ and $\mathcal{BV}_{k-1}\!\subset\!\mathcal{BV}_{k}$.  It is easy to see that the value of a type $k$ invariant on a $k$-singular link depends only on the placement of singularities and not on the way $\coprod_m\R$ is mapped to $\R^3$.  This information is encoded by \emph{chord diagrams of order $k$}, which are our usual diagrams except they have exactly $2k$ external vertices paired off by $k$ chords and no internal vertices.  Denote by $\mathcal{CLD}^{k,0}$ and $\mathcal{CBD}^{k,0}$ the subspaces of $\mathcal{LD}^{k,0}$ and $\mathcal{BD}^{k,0}$, respectively, generated by chord diagrams of order $k$.

Now consider the \emph{four-term} (\emph{4T}) relation pictured in Figure \ref{Fig:4TRelation} and recall the 1T relation from Figure \ref{Fig:1TRelation}.  Define spaces of \emph{weight systems of order $k$}, denoted by $\mathcal{LW}_k$ and $\mathcal{BW}_k$, to be 
$$
\mathcal{LW}_k=(\mathcal{CLD}^{k,0}/(4T, 1T))^*\ \ \ \ \text{and}\ \ \ \ \mathcal{BW}_k=(\mathcal{CBD}^{k,0}/4T)^*
$$
where $*$ denotes the dual vector space.

\begin{rem}
As before, it is not necessary to use the 1T relation in the braid case since those diagrams do not contain chords on single strands.
\end{rem}
%%\raisebox{-.15cm}{
%\begin{figure}
%  \centering
%  \fbox{\includegraphics[height=40mm]{4TRelation.jpg}}
%  \caption{$4T$ relation.   No other edges connect to the pictured vertices and the diagrams are same outside pictured portions.  In the case of $\Lk_{m}^{3}$, the three subsegments need not lie on distinct segments.  In particular, when they are parts of a single segment, we get the usual $4T$ relation for knots.
%%   but they cannot all be the same in the case of $\HLk_{m}^{3}$ and $\Br_{m}^{3}$. If $\K_1$ and $\K_2$ are distinct, then first two pictures are same.
%   }
%  \label{Fig:4TRelation}
%\end{figure}
%%}

\begin{figure}[h]
\input{4trelation.pstex_t}
\caption{$4T$ relation.   No other edges connect to the pictured vertices and the diagrams are same outside pictured portions.  In the case of $\Lk_{m}^{3}$, the three subsegments need not lie on distinct segments.  In particular, when all the external vertices are on a single segment, we get the usual $4T$ relation for knots.
%
%   but they cannot all be the same in the case of $\HLk_{m}^{3}$ and $\Br_{m}^{3}$. If $\K_1$ and $\K_2$ are distinct, then first two pictures are same.
}
\label{Fig:4TRelation}
\end{figure}

The following was established by Bar-Natan, \cite[Theorem 6]{BN:Vass}, in case of knots (i.e~for the diagram complex $\mathcal{D}$).  The proof is identical for the case of links and braids, so that we get

\begin{thm}\label{T:DiagramsIsomorphism} There are isomorphisms
$$
H^0(\LD^{k,0})\cong \mathcal{LW}_k \ \ \ \text{and}\ \ \    H^0(\BD^{k,0})\cong \mathcal{BW}_k.
$$

%

%$$
%H^0(\mathcal{D}^{k,0})\cong \W_k.
%$$
\end{thm}

We now have the main theorem of this section.

\begin{thm}\label{T:UniversalFiniteType} 
The maps $\overline{I}_{\Lk}^*$ and $\overline{I}_{\Br}^*$ give isomorphisms
$$
\mathcal{LW}_k\cong\mathcal{LV}_k/\mathcal{LV}_{k-1}  \ \ \ \text{and}\ \ \ \mathcal{BW}_k\cong\mathcal{BV}_k/\mathcal{BV}_{k-1}
$$
between spaces of weight systems and finite type invariants of the two link spaces.
\end{thm}

\begin{proof}  We consider the case of $\overline{I}_{\Lk}^*$.  The braid case is identical.
%Continuing to denote by $\calX$ either of the spaces $\Lk_{m}^{3}$ or $\Br_{m}^{3}$; by $\mathcal{D}^{k,o}$ either of the trivalent diagram subspaces of  $\LD^{k,0}$ and $\BD^{k,0}$; and by $I$ either of the chain maps $\overline{I}_{\Lk}$ and $I_{\Br}$, 
We have from \refT{MapsForn=3} a map
\begin{align*}
\overline{I}_{\Lk}^*\colon \mathcal{LW}_k & \longrightarrow H^0(\Lk^3_m)  \\
\end{align*}
(Even though $W\in \mathcal{LW}_k$ is now a priori defined only on chord diagrams, it can be extended uniquely to trivalent diagrams via isomorphism from \refT{DiagramsIsomorphism}.)
We want to prove that this map is an isomorphism onto those finite type $k$ elements in $H^0(\Lk^3_m)$ which do not come from finite type $k-1$ elements.  

To see that the map lands in $\mathcal{LV}_k/\mathcal{LV}_{k-1}$, consider a $k$-singular link $LK_k$  with singularities as prescribed by some chord diagram $\Gamma_0\in\mathcal{CLD}^{k,o}$ (i.e.~chords connect those points on the segments of $\Gamma_0$ which, after those segments have been identified with $m$ copies of $\R$ and mapped to $\R^3$, make up the singularities).  Now consider all the resolutions $LK_{k}^{1}$, $LK_{k}^{2}$, ..., $LK_{k}^{2^k}$ of $LK_k$.  We want to show that the values of $\overline{I}_{\Lk}^*$ cancel or vanish over all $\Gamma\in \mathcal{LD}^{k,o}$ and all resolutions.

First observe that we can choose the resolutions so they only differ in $k$ disjoint balls of arbitrarily small radius.   Then the value of $(\overline{I}_{\Lk})_{\Gamma_0}$ on each of the $2^k$ resolutions of $LK_k$ is $1$ or $-1$ because each of the $k$ maps between $2k$ configuration points on the resolutions can point either to the north pole or the south pole inside each ball.  Those integrals thus cancel out.  Note that this would not happen for any chord diagram with fewer than $k$ chords.

The argument for diagrams that are different from $\Gamma_0$ is essentially that the vectors can never point to the poles inside the resolution balls at the same time.  So at least one vector must point from one ball to the other.  The difference, over all resolutions, can then be made arbitrarily small.  

What remains to see is that $\overline{I}_{\Lk}^*$ is an isomorphism onto $\mathcal{LV}_k/\mathcal{LV}_{k-1}$.  This is accomplished by noticing that there is a map 
\begin{align*}
\mathcal{LV}_k & \longrightarrow \mathcal{LW}_k \\
V  &  \longmapsto  f
\end{align*}
where $f$ is given by 
$$
f(\Gamma)=V(LK_{\Gamma}).
$$
Here $LK_{\Gamma}$ is any link with $k$ singularities as prescribed by the chord diagram $\Gamma$ (remember that $\Gamma$ has $k$ chords pairing off $2k$ external vertices; these points are to form singularities after the $m$ copies of $\R$, which correspond to diagram segments, are mapped to $\R^3$).  This is a well-defined map because, as noted before, the value of a type $k$ invariant on a $k$-singular link only depends on the placement of singularities.  The kernel of this map is by definition precisely $\mathcal{LV}_{k-1}$.

The claim now is that this map and $\overline{I}_{\Lk}^*$ are inverses.  The proof of this is identical to the case of knots, which is Theorem 5.3 in \cite{V:SBT} (which contains more details about the first part of the proof as well), and does not amount to much more than unravelling of the definitions.  The translation to the case of links from the case of knots is thus left to the reader.
\end{proof}

\bibliographystyle{amsplain}

\bibliography{/Users/ivolic/Desktop/Papers/Bibliography}

\end{document}